\documentclass[10pt,a4paper,twoside]{article}

\usepackage{amssymb}
\usepackage{amsthm}
\usepackage{amsmath}

\usepackage{enumerate}
\usepackage{hyperref}

\newtheorem{thmA}{Theorem}
\newtheorem{thm}{Theorem}
\newtheorem{lemA}{Lemma}
\newtheorem{lem}{Lemma}
\newtheorem{rem}{Remark}

\newcommand*{\half}{\frac12}
\newcommand*{\quart}{\frac14}
\newcommand*{\e}{\mbox{e}}
\newcommand*{\lb}{\left\{}
\newcommand*{\rb}{\right\}}
\newcommand*{\RR}{\mathbb{R}}
\newcommand*{\ZZ}{\mathbb{Z}}
\newcommand*{\PP}{\mathbf{P}}
\newcommand*{\EE}{\mathbf{E}}
\newcommand*{\Var}{\mathbf{Var}}
\newcommand*{\II}{\mathbf{I}}
\newcommand*{\di}{\, \mathrm{d} }
\newcommand*{\qv}[2][M]{\langle #1,#1 \rangle_{#2}}

\begin{document}


\title{Strong approximation of continuous local martingales by simple
random walks}

\author{Bal\'azs Sz\'ekely\footnote{Research
supported by the HSN laboratory of BUTE.} and
Tam\'as Szabados\footnote{Corresponding author, address:
Department of Mathematics, Budapest University of Technology and
Economics, M\H{u}egyetem rkp. 3, H \'ep. V em. Budapest, 1521,
Hungary, e-mail: szabados@math.bme.hu, telephone: (+36 1) 463-1111/ext. 5907,
fax: (+36 1) 463-1677} \footnote{Research
supported by the French--Hungarian intergovernmental grant
``Balaton'' F-39/2000.} \\
Budapest University of Technology and Economics}

\date{}

\maketitle


\begin{abstract}
The aim of this paper is to represent any continuous local martingale
as an almost sure limit of a nested sequence of simple, symmetric random
walks, time changed by a discrete quadratic variation process. One
basis of this is a similar construction of Brownian
motion. The other major tool is a representation of continuous local
martingales given by Dambis, Dubins and Schwarz (DDS) in terms of
Brownian motion time-changed by the quadratic variation. Rates of
convergence (which are conjectured to be nearly optimal in the given
setting) are also supplied. A necessary and sufficient condition for
the independence of the random walks and the discrete time changes
or, equivalently, for the independence of the DDS Brownian motion and
the quadratic variation is proved to be the symmetry of increments of
the martingale given the past, which is a reformulation of an earlier
result by Ocone \cite{Oco93}.
\end{abstract}


\renewcommand{\thefootnote}{\alph{footnote}}
\footnotetext{ 2000 \emph{MSC.} Primary 60F15, 60G44. Secondary 60H05.}

\section{Introduction}

The present authors are convinced that both for theoretical and
practical reasons, it is useful to search for strong (i.e. pathwise,
almost sure) approximations of stochastic processes by simple random
walks (RWs). The prototype of such efforts was the construction of
Brownian motion (BM) as an almost sure limit of simple RW paths,
given by Frank Knight in 1962 \cite{Kni62}. Later this
construction was simplified and somewhat improved by P\'al R\'ev\'esz
\cite{Rev90} and then by one of the authors \cite{Sza96}. It is
interesting that this method is asymptotically equivalent to Skorohod
embedding of a nested sequence of RWs into BM, \cite[Theorem
4]{Sza96}. This elementary approach also led to a strong
approximation of It\^o integrals of smooth functions of BM
\cite[Theorem 6]{Sza96}. We mention that this RW construction was
extended to fractional BMs as well \cite{Sza01}.

This paper discusses a generalization to continuous local martingales
$M$. Beside the above-mentioned RW construction of BM, the other
major tool applied by the paper is a representation of continuous
local martingales by Brownian motion, time-changed by the quadratic
variation, given by Dambis \cite{Dam65} and Dubins -- Schwarz
\cite{DuS65} (DDS). Theorem \ref{th:qvar} shows that the quadratic
variation process $\qv{}$ can be almost surely uniformly approximated by a
discrete quadratic variation processes $N_m$ which are based on
stopping times of a Skorohod-type embedding of nested simple RWs into
$M$. This corresponds to an earlier similar result by Karandikar
\cite{Kar83}. Theorems \ref{th:approx} and \ref{th:approxNm} give an
approximation of $M$ by a nested sequence of RWs $B_m$, time-changed
by $\qv{}$ and $N_m$, respectively. The approximations almost surely
uniformly converge on bounded intervals. Rates of convergence (which
conjectured to be nearly optimal in the given setting) are also
supplied.

It is important to note that the DDS Brownian motion $W$ and the
quadratic variation $\qv{}$ are not independent in general, just
like the approximating RW $B_m$ and the discrete quadratic
variation $N_m$. Since this could be a hindrance both in the
theory and applications, a necessary and sufficient condition is
given for the independence in Theorem \ref{th:indep}. Namely, the
approximating RW $B_m$ and the discrete quadratic variation $N_m$
(and so $W$ and $\qv{}$) are independent if and only if $M$ has
symmetric increments given the past. This is a reformulation of an
earlier result by Ocone \cite{Oco93}, see \cite{DEY93} and
\cite{VoY00} as well.

Possible applications of the discrete approximations discussed in
this paper include (a) generating continuous local martingales for
which one has a suitable method to generate the discrete quadratic
variation process $N_m$, especially when the independence
mentioned in the previous paragraph holds and (b) giving an
alternative definition of stochastic integration with almost
surely converging sums. (In a forthcoming paper the present
authors are going to take up the second topic, extending some
results of \cite{Sza96}.)

\section{Random walks and the Wiener process}

A main tool of the present paper is an elementary construction of the
Wiener process (= BM). The specific construction we are going to use
in the sequel, taken from \cite{Sza96}, is based on a nested sequence
of simple random walks that uniformly converges to the Wiener process
on bounded intervals with probability 1. This will be called \emph{RW
construction} in the sequel. One of our intentions in this paper is
to extend the underlying \emph{``twist and shrink''} algorithm to
continuous local martingales.

We summarize the major steps of the RW construction here, see
\cite{Sza01} as well. We start with an infinite matrix of i.i.d. random
variables $X_m(k)$, $\PP \lb X_m(k)= \pm 1 \rb = 1/2$ ($m\ge 0$,
$k\ge 1$), defined on the same underlying probability space
$(\Omega,\mathcal{F},\PP)$.  Each row of this matrix is a basis
of an approximation of the Wiener process with a dyadic step size
$\Delta t=2^{-2m}$ in time and a corresponding step size $\Delta
x=2^{-m}$ in space, illustrated by the next table.

\begin{table}[h]
\caption{The starting setting for the RW construction of BM}
\[
\begin{array}{|c|c|l|l|}
\hline
\Delta t &\Delta x & \mbox{i.i.d. sequence} & \mbox{RW} \\
\hline
1&1& X_0(1), X_0(2), X_0(3), \dots & S_0(n) = \sum_{k=1}^n X_0(k) \\
2^{-2}&2^{-1}& X_1(1), X_1(2), X_1(3), \dots & S_1(n) = \sum_{k=1}^n
X_1(k) \\
2^{-4}&2^{-2}& X_2(1), X_2(2), X_2(3), \dots & S_2(n) = \sum_{k=1}^n
X_2(k) \\
\vdots & \vdots & \vdots & \vdots \\
\hline
\end{array}
\]
\end{table}

The second step of the construction is \emph{twisting}. From the
independent random walks we want to create dependent ones so that
after shrinking temporal and spatial step sizes, each consecutive RW
becomes a refinement of the previous one.  Since the spatial unit will
be halved at each consecutive row, we define stopping times by
$T_m(0)=0$, and for $k\ge 0$,
\[
T_m(k+1)=\min \{n: n>T_m(k), |S_m(n)-S_m(T_m(k))|=2\} \qquad (m\ge 1)
\]
These are the random time instants when a RW visits even integers,
different from the previous one. After shrinking the spatial unit by
half, a suitable modification of this RW will visit the same integers in
the same order as the previous RW. We operate here on each point
$\omega\in\Omega$ of the sample space separately, i.e. we fix a sample
path of each RW. We define twisted RW{}s $\tilde{S}_m$ recursively for
$k=1,2,\dots$ using $\tilde{S}_{m-1}$, starting with
$\tilde{S}_0(n)=S_0(n)$ $(n\ge 0)$. With each fixed $m$ we proceed for
$k=0,1,2,\dots$ successively, and for every $n$ in the corresponding
bridge, $T_m(k)<n\le T_m(k+1)$. Any bridge is flipped if its sign
differs from the desired:
\[
\tilde{X}_m(n)=\left\{
\begin{array}{rl}
 X_m(n)& \mbox{ if } S_m(T_m(k+1)) - S_m(T_m(k))
= 2\tilde X_{m-1}(k+1), \\
- X_m(n)& \mbox{ otherwise,}
\end{array}
\right.
\]
and then $\tilde{S}_m(n)=\tilde{S}_m(n-1)+\tilde{X}_m(n)$. Then
$\tilde{S}_m(n)$ $(n\ge 0)$ is still a simple symmetric random walk
\cite[Lemma 1]{Sza96}. The twisted RWs have the desired refinement
property:
\[
\half \tilde{S}_m(T_m(k))=\tilde{S}_{m-1}(k) \qquad (m\ge 1, k\ge 0).
\]

The last step of the RW construction is \emph{shrinking}. The sample
paths of $\tilde{S}_m(n)$ $(n\ge 0)$ can be extended to continuous
functions by linear interpolation, this way one gets $\tilde{S}_m(t)$
$(t\ge 0)$ for real $t$. Then we define the $mth$ \emph{approximating
RW} by
\[
\tilde{B}_m(t)=2^{-m}\tilde{S}_m(t2^{2m}).
\]
Using the definition of $T_m$ and $\tilde{B}_m$ we also get the general
\emph{refinement property}
\begin{equation}
\tilde{B}_{m+1}\left(T_{m+1}(k)2^{-2(m+1)}\right) = \tilde{B}_m
\left( k2^{-2m}\right) \qquad (m\ge 0,k\ge 0).
\label{eq:refin}
\end{equation}
Note that a refinement takes the same dyadic values in the same order
as the previous shrunken walk, but there is a \emph{time lag} in general:
\begin{equation} T_{m+1}(k)2^{-2(m+1)} - k2^{-2m} \ne 0 .
\label{eq:tlag}
\end{equation}

Then we quote some important facts from \cite{Sza96} about the above
RW construction that will be used in the sequel. These will be stated
in somewhat stronger forms but can be read easily from the proofs in
the cited reference, cf. Lemmas 2-4 and Theorem 3 there.

\begin{lemA} \label{le:ldi}
Suppose that $X_1,X_2,\dots ,X_N$ is an i.i.d. sequence of random
variables, $\EE(X_k)=0$, $\Var(X_k)=1$, and their
moment generating function is finite in a neighborhood of 0. Let
$S_j = X_1 + \dots + X_j$, $1\le j\le N$. Then for any $C>1$ and $N \ge
N_0(C)$ one has
\[
\PP \lb \sup_{1 \le j \le N} |S_j| \ge (2CN \log N)^{\half} \rb \le 2N^{1-C} .
\]
\end{lemA}
We mention that this basic fact, that appears in the above-mentioned
reference \cite{Sza96}, essentially depends on a large deviation theorem.

We have a more convenient result in a special case of Hoeffding's inequality,
cf. \cite{Hoef63}. Let $X_1,
X_2, \dots $ be a sequence of bounded i.i.d. random variables, such that
$b_i \le X_i \le a_i$, and let $S_n = \sum_{i=1}^n X_i$. Then by
Hoeffding's inequality, for any $x > 0$ we have
\[
\PP \lb |S_n - \EE (S_n)| \ge x \left(\frac14 \sum_{i=1}^n
(a_i - b_i)^2\right)^{\frac12} \rb \le 2 \: \e^{-\frac{x^2}{2}} .
\]
If $\EE (X_i) = 0$ and $b_i = -a_i$ here, then $\frac14 \sum_{i=1}^n
(a_i - b_i)^2 = \sum_{i=1}^n a_i^2 = \Var (S_n)$ if and only if
$X_i = a_i X'_i$, where $\PP \lb
X'_i = \pm 1 \rb = \half$, $1 \le i \le n$.

Thus if $S = \sum_r a_r X'_r$, where not all $a_r$ are zero and $\Var (S)
= \sum_r a_r^2 < \infty$, we get
\begin{equation} \label{eq:ldi2}
\PP \lb |S| \ge x \left( \Var (S)\right) ^{\half} \rb \le 2
\: \e^{-\frac{x^2}{2}} \qquad (x \ge 0) .
\end{equation}
The summation above may extend either
to finitely many or to countably many terms. Let
$S_1, S_2, \dots , S_N$ be arbitrary sums of the
above type: $S_k = \sum_r a_{k r} X'_{k r}$, $\PP \lb
X'_{k r} = \pm 1 \rb = \half$, $1 \le k \le N$, where $X'_{k r}$ and
$X'_{l s}$ can be dependent when $k \ne l$. Then by the inequality
(\ref{eq:ldi2}) we obtain the following analog of Lemma~\ref{le:ldi}:
for any $C >1$ and $N \ge 1$,
\begin{eqnarray}
\lefteqn{\PP \lb \sup_{1\le k \le N} |S_k|
\ge (2C\log N)^{\half} \sup_{1\le k \le N} \left( \Var (S_k)\right)
^\half \rb} \nonumber \\
&\le& \sum_{k=1}^N \PP \lb |S_k| \ge (2C\log N \: \Var (S_k))
^\half \rb \le 2 N \e^{-C \log N} = 2 N^{1-C}. \label{eq:ldp2}
\end{eqnarray}

Lemma \ref{le:ldi} easily implies that
the time lags (\ref{eq:tlag}) are uniformly small if $m$ is large
enough.
\begin{lemA} \label{le:tlag}
 For any $K>0$, $C>1$, and for any $m\ge m_0(C)$, we have
\begin{eqnarray*}
\PP \lb \sup_{0\le k2^{-2m}\le K} |T_{m+1}(k)2^{-2(m+1)}
-k2^{-2m}| \ge \left(\frac32 C K \log{}_*K\right)^{\half}
m^{\half} 2^{-m} \rb \\
\le 2(K2^{2m})^{1-C} ,
\end{eqnarray*}
where $\log{}_*x=\max \{1,\log x\}$.
\end{lemA}

This lemma and the refinement property (\ref{eq:refin}) implies
the uniform closeness of two consecutive approximations if $m$ is
large enough.

\begin{lemA} \label{le:refin}
For any $K>0$, $C>1$, and for any $m\ge m_1(C)$, we have
\begin{eqnarray*}
\PP \lb \sup_{0\le k2^{-2m} \le K} |\tilde{B}_{m+1}(k2^{-2m})
-\tilde{B}_m(k2^{-2m})| \ge K_*^{\quart}(\log_*K)^{\frac34}
m2^{-\frac{m}{2}} \rb \\
\le 3(K2^{2m})^{1-C} ,
\end{eqnarray*}
where $K_* = \max\{1, K\}$.
\end{lemA}

Based on this lemma, it is not difficult to show the following
convergence result.

\begin{thmA} \label{th:Wiener}
The shrunken RWs $\tilde{B}_m(t)$ $(t\ge 0, m=0,1,2,\dots)$ almost surely
uniformly converge to a Wiener process $W(t)$ $(t\ge 0)$ on any compact
interval $[0,K]$, $K>0$. For any $K>0$, $C\ge 3/2$, and for any $m\ge m_2(C)$,
we have
\[
\PP \lb \sup_{0 \le t \le K} |W(t) - \tilde{B}_m(t)| \ge
K_*^{\quart}(\log_*K)^{\frac34} m 2^{-\frac{m}{2}} \rb \le
6(K2^{2m})^{1-C} .
\]
\end{thmA}

Now taking $C=3$ in Theorem \ref{th:Wiener} and using the Borel--Cantelli
lemma, we get
\[
\sup_{0 \le t \le K} |W(t) - \tilde{B}_m(t)| < O(1) m
2^{-\frac{m}{2}} \qquad \mbox{a.s.} \qquad (m \to \infty)
\]
and
\[
\sup_{0 \le t \le K} |W(t) - \tilde{B}_m(t)| <
K^{\quart} (\log K)^{\frac34} \qquad \mbox{a.s.} \qquad (K \to \infty)
\]
for any $m$ large enough, $m \ge m_2(3)$.

Next we are going to study the properties of another nested
sequence of random walks, obtained by Skorohod embedding. This
sequence is not identical, though asymptotically equivalent to the
above RW construction, cf. \cite[Theorem 4]{Sza96}. Given a Wiener
process $W$, first we define the stopping times which yield the
Skorohod embedded process $B_m(k2^{-2m})$ into $W$. For every
$m\ge 0$ let $s_m(0)=0$ and
\begin{equation}
s_m(k+1)=\inf{}\{s: s > s_m(k), |W(s)-W(s_m(k))|=2^{-m}\} \qquad
(k \ge 0). \label{eq:Skor1}
\end{equation}
With these stopping times the embedded process by definition is
\begin{equation}
B_m(k2^{-2m}) = W(s_m(k)) \qquad (m\ge 0, k\ge 0). \label{eq:Skor2}
\end{equation}
This definition of $B_m$ can be extended to any real $t \ge 0$ by
pathwise linear interpolation. The next lemma describes some useful
facts about the relationship between $\tilde{B}_m$ and $B_m$. These
follow from \cite[Lemmas 5,7 and Theorem 4]{Sza96}, with some minor
modifications.

In general, roughly saying, $\tilde{B}_m$ is more useful when
someone wants to generate stochastic processes from scratch, while
$B_m$ is more advantageous when someone needs a discrete
approximation of given processes, like in the case of stochastic
integration.

\begin{lemA} \label{le:equid}
For any $C\ge 3/2$, $K>0$, take the following subset of the sample
space:
\begin{equation}
A_m=\lb \sup_{n > m} \: \sup_{0\le k2^{-2m}\le K} |2^{-2n}
T_{m,n}(k) - k2^{-2m}| < 6(CK_* \log_*K)^{\half} m^{\half} 2^{-m} \rb ,
\label{eq:A_m}
\end{equation}
where $T_{m,n}(k) = T_n \circ T_{n-1}
\circ \cdots \circ T_m(k)$ for $n>m\ge 0$ and $k\ge 0$.
Then for any $m\ge m_3(C)$,
\[
\PP \lb A_m^c \rb \le 4(K2^{2m})^{1-C} .
\]
Moreover, $\lim_{n\to \infty} 2^{-2n} T_{m,n}(k) = t_m(k)$ exists almost
surely and on the set $A_m$ we have
\[
\tilde{B}_m(k2^{-2m}) = W(t_m(k)) \qquad (0 \le k2^{-2m} \le K) ,
\]
cf. (\ref{eq:Skor2}). Further, on $A_m$ except for a zero
probability subset, $s_m(k) = t_m(k)$ and
\begin{equation}
\sup_{0\le k2^{-2m}\le K} |s_m(k)-k2^{-2m}| \le 6(CK_* \log_* K)
^{\half} m^{\half} 2^{-m} \qquad (m\ge m_3(C)) . \label{eq:sm}
\end{equation}
\end{lemA}

If the Wiener process is built by the RW construction described above
using a sequence $\tilde{B}_m$ ($m \ge 0$) of nested RWs and then one
constructs the Skorohod embedded RWs $B_m$ ($m \ge 0$), it is natural
to ask what the approximating properties of the latter are. The answer
described by the next theorem is that they are essentially the same as
the ones of $\tilde{B}_m$, cf. Theorem \ref{th:Wiener}.

\begin{thm} \label{th:Wienerm}
For every $K>0$, $C \ge 3/2$ and $m\ge m_3(C)$ we have
\[
\PP \lb \sup_{0\le t\le K} \left| W(t) - B_m(t) \right|
\ge K_*^{\quart}(\log_*K)^{\frac34} m 2^{-\frac{m}{2}} \rb
\le 10(K2^{2m})^{1-C} .
\]
\end{thm}

\begin{proof}
By the triangle inequality,
\[
\sup_{0\le t\le K} \left| W(t) - B_m(t) \right| \le
\sup_{0\le t\le K} \left| W(t) - \tilde{B}_m(t) \right|
+ \sup_{0\le t\le K} \left| \tilde{B}_m(t) - B_m(t) \right| .
\]
By Lemma \ref{le:equid} and equation (\ref{eq:Skor2}), on the set $A_m$
defined by (\ref{eq:A_m}) we have
\[ \tilde{B}_m(k2^{-2m}) = W(s_m(k)) = B_m(k2^{-2m}), \]
except for a zero probability subset when $m \ge m_3(C)$. Since both
$\tilde{B}_m(t)$ and $B_m(t)$ are obtained by pathwise linear
interpolation based on the vertices at $k2^{-2m} \in [0, K]$, they are
identical on $A_m$, except for a zero probability subset of it when $m
\ge m_3(C)$. Thus
\begin{eqnarray*}
\lefteqn{\PP \lb \sup_{0\le t\le K} \left| W(t) - B_m(t) \right|
\ge K_*^{\quart}(\log_*K)^{\frac34} m 2^{-\frac{m}{2}} \rb} \\
&\le& \PP \lb A_m^c \rb + \PP \lb \sup_{0\le t\le K} \left|
W(t) - \tilde{B}_m(t) \right|
\ge K_*^{\quart}(\log_*K)^{\frac34} m 2^{-\frac{m}{2}} \rb
\end{eqnarray*}
Then by Theorem \ref{th:Wiener} and Lemma \ref{le:equid} we
get the statement of the theorem.
\end{proof}

\section{The basic approximation results}

Beside the RW construction of standard Brownian motion, the other
main tool applied in this paper is a theorem of Dambis (1965) and
Dubins--Schwarz (1965) and an extension of it, cf. Theorems
\ref{DDS1} and \ref{DDS2} below. Briefly saying, these theorems state
that any continuous local martingale $(M(t), t \ge 0)$ can be
transformed into a standard Brownian motion by time-change. Then
somewhat loosely speaking, the resulting Brownian motion takes on the
same values in the same order as $M(t)$, only the corresponding time
instants may differ. These and other necessary matters about
continuous local martingales will be taken from and discussed in the
style of \cite{Yor99} in the sequel.

Below it is supposed that an increasing family of sub-$\sigma$-algebras
$(\mathcal{F}_t, t \ge 0)$ is given in the probability space $(\Omega,
\mathcal{F}, \PP)$ and the given continuous local martingale
$M$ is adapted to it.

In the case of a continuous local martingale $M(t)$ vanishing at 0
its quadratic variation $\qv{t}$ is a process with almost surely  continuous
and non-decreasing sample paths vanishing at 0.
This will be one of the two time-changes we are going to use in the
sequel. The other one is a quasi-inverse of the quadratic variation:
\begin{equation} \label{eq:T}
T_s=\inf \{ t: \qv{t} > s \}  ,
\end{equation}
where $\inf (\emptyset) = \infty$ by definition. Then the sample paths
of the process $T_s$ are almost surely increasing, but only
right-continuous, since such a path has a jump at any value where the
quadratic variation has a constant level-stretch. Beside this, $T_s$
may be infinite valued. The duality  between the two time-changes
is expressed by $\qv{t} = \inf \{ s: T_s > t \}$.
Observe that $T_s$ cannot have constant level-stretches since this would
imply jumps for $\qv{t}$. Also
the continuity of $\qv{t}$ gives that $\qv{T_s} = s$ ($s \ge 0$), while
we have only $ T_{\qv{t}} \ge t$ ($t \ge 0$) in the opposite
direction. It is clear that
\begin{equation}
\qv{t} <  \: s  \Longrightarrow  t <  \: T_s,
\mbox{ but }  t <  \: T_s  \Longrightarrow  \qv{t} \le  \: s ,
\label{eq:tchange1}
\end{equation}
while
\begin{equation}
\qv{t} \: \le, \: \ge, \: > \:  s  \iff  t \: \le, \: \ge,\: > \: T_s ,
\label{eq:tchange2}
\end{equation}
respectively.

\begin{thmA} \label{DDS1} \cite[V (1.6), p.181]{Yor99}
If $M$ is a continuous $(\mathcal{F}_t)$-local martingale
vanishing at 0 and such that $\qv{\infty}=\infty$ a.s., then
$W(s)=M(T_s)$ is an $(\mathcal{F}_{T_s})$-Brownian motion and
$M(t)=W(\qv{t})$.
\end{thmA}

Similar statement is true when $\qv{\infty}<\infty$ is possible. Note
that on the set $\lb \qv{\infty}<\infty \rb$ the limit $M(\infty)
=\lim_{t\to\infty} M(t)$ exists with probability 1, cf. \cite[IV
(1.26), p. 131]{Yor99}.

\begin{thmA} \label{DDS2} \cite[V (1.7), p.182]{Yor99}
If $M$ is a continuous $(\mathcal{F}_t)$-local martingale
vanishing at 0 and such that $\qv{\infty}<\infty$ with positive
probability, then there exists an enlargement
$(\widetilde{\Omega},\widetilde{ \mathcal{F}}_t, \widetilde{\PP})$
of $(\Omega,\mathcal{F}_{T_t},\PP)$ and a Wiener process
$\widetilde{\beta}$ on $\widetilde{\Omega}$ independent of $M$
such that the process
\[
W(s)=\left\{
\begin{array}{ll}
M(T_s) &\mbox{if } s < \qv{\infty} , \\
M(\infty) + \widetilde{\beta}(s - \qv{\infty}) &\mbox{if }
s \ge \qv{\infty}
\end{array}
\right.
\]
is a standard Brownian motion and $M(t)=W(\qv{t})$ for $t \ge 0$.
\end{thmA}

From now on, $W$ will always refer to the Wiener process obtained
from $M$ by the above time-change, the so-called \emph{DDS Wiener
process (or DDS Brownian motion)} of $M$.

Now Skorohod-type stopping times can be defined for $M$, similarly as
for $W$ in (\ref{eq:Skor1}). For $m \ge 0$, let $\tau_m(0)=0$ and
\begin{equation}
\tau_m(k+1)=\inf \left\{t: t > \tau_m(k), \left| M(t) - M(\tau_m(k))
\right| = 2^{-m} \right\} \qquad (k \ge 0) . \label{eq:tau}
\end{equation}
The $(m+1)$st stopping time sequence is a refinement of the $m$th
in the sense that $\left(\tau_m(k)\right)_{k=0}^{\infty}$ is a
subsequence of $\left(\tau_{m+1}(j)\right)_{j=0}^{\infty}$ so that
for any $k\ge 0$ there exist $j_1$ and $j_2$, $\tau_{m+1}(j_1) =
\tau_m(k)$ and $\tau_{m+1}(j_2) = \tau_m(k+1)$, where the
difference $j_2 - j_1 \ge 2$, even.

\begin{lem} \label{le:SkorM}
With the stopping times defined by (\ref{eq:tau}) from a
continuous local martingale $M$ one can directly obtain the
sequence of shrunken RWs that almost surely converges to the DDS Wiener
process $W$ of $M$, cf. (\ref{eq:Skor2}):
\[
B_m(k2^{-2m}) = W(s_m(k)) = M(\tau_m(k)) , \qquad s_m(k)
= \qv{\tau_m(k)}
\]
[but $\tau_m(k) \le T_{s_m(k)}$ ], where for $m\ge 0$, the non-negative integer $k$ is
taking values (depending on $\omega$) until $s_m(k) \le
\qv{\infty}$.
\end{lem}

\begin{proof}
By Theorems \ref{DDS1} and \ref{DDS2} it follows that
$W(\qv{\tau_m(k)}) = M(\tau_m(k))$. This implies that $s_m(k) \le
\qv{\tau_m(k)}$. Then consider first the case $k=1$. If $s_m(1) <
\qv{\tau_m(1)}$ held, then $T_{s_m(1)} < \tau_m(1)$ would follow
by (\ref{eq:tchange2}), and this would lead to a contradiction
because $M(T_{s_m(1)}) = W(s_m(1)) = \pm 2^{-m}$. For values $k >
1$, induction with a similar argument can show the statement of
the lemma.
\end{proof}

Below $B_m$ will always denote the sequence of shrunken RWs
defined by Lemma \ref{le:SkorM}.

Our next objective is to show that the quadratic variation of $M$ can
be obtained as an almost sure limit of a point process related to the
above stopping times that we will call \emph{a discrete quadratic
variation process}:
\begin{eqnarray}
N_m(t) &=& 2^{-2m} \# \{r: r>0, \tau_m(r) \le t \}  \label{eq:N_m} \\
&=& 2^{-2m} \# \{r: r>0, s_m(r) \le \qv{t} \} \qquad (t\ge 0) .
\nonumber
\end{eqnarray}
Clearly, the paths of $N_m$ are non-decreasing pure jump functions,
the jumping times being exactly the stopping times
$\tau_m(k)$. Moreover, $N_m\left(\tau_m(k)\right) = k2^{-2m}$ and the
magnitudes of jumps are constant $2^{-2m}$ when $m$ is fixed.

\begin{lem} \label{le:qvar}
Let M be a continuous local martingale vanishing at 0, let $\qv{}$ be
the quadratic variation, $T$ be its quasi-inverse (\ref{eq:T}), and
$N_m$ be the discrete quadratic variation defined in (\ref{eq:N_m}).
Fix $K > 0$ and take a sequence $a_m = O(m^{-2-\epsilon} 2^{2m}) K$ with
some $\epsilon > 0$, where $a_m \ge K \lor 1$ for any $m \ge 1$
($x \lor y = \max(x, y)$, $x \wedge y = \min(x, y)$).

(a) Then for any $C\ge 3/2$ and $m\ge m_4(C)$ we have
\begin{eqnarray*}
\PP \lb \sup_{0 \le t \le K} \left| \qv{t}\wedge
a_m - N_m(t \wedge T_{a_m}) \right| \ge 12(C a_m \log_*a_m)^{\half}
m^{\half} 2^{-m} \rb  \\
\le 3(a_m 2^{2m})^{1-C} .
\end{eqnarray*}

(b) Suppose that the quadratic variation satisfies the following
tail-condition: a sequence $(a_m)$ fulfilling the above assumptions
can be chosen so that
\begin{equation}
\PP \lb \qv{t} > a_m \rb \le D(t) m^{-1-\epsilon} ,
\label{eq:tail}
\end{equation}
where D(t) is some finite valued function of $t \in \RR_+$.
Then for any $C\ge 3/2$ and $m\ge m_4(C)$ it follows that
\begin{eqnarray*}
\PP \lb \sup_{0 \le t \le K} \left| \qv{t} - N_m(t) \right|
\ge 12(C a_m \log_*a_m)^{\half} m^{\half} 2^{-m} \rb \\
\le 3(a_m 2^{2m})^{1-C} + D(K) m^{-1-\epsilon} .
\end{eqnarray*}
\end{lem}

\begin{proof}
The basic idea of the proof is that the Skorohod stopping times of a
Wiener process are asymptotically uniformly distributed as shown by
(\ref{eq:sm}), while the case of a continuous local martingale can be
reduced to the former by the DDS representation, cf. Lemma \ref{le:SkorM}.

Introduce the abbreviation $h_{a,m} = 11.1 \left(C a_m \log_*a_m
\right)^{\half} m^{\half} 2^{-m}$. Then $h_{a,m}$ $ = O(m^{-\epsilon})
\to 0$ as $m \to \infty$. We
need a truncation here using the sequence $a_m$, since the
quadratic variation $\qv{t}$ is not a bounded
random variable in general. By (\ref{eq:tchange2}) and (\ref{eq:N_m}),
\begin{eqnarray*}
N_m(t \wedge T_{a_m}) &=& 2^{-2m} \# \{r: r>0, \tau_m(r) \le t
\wedge T_{a_m} \} \\
&=& 2^{-2m} \# \{r: r>0, s_m(r) \le \qv{t} \wedge a_m \} .
\end{eqnarray*}
On the event
\[ A_{a,m} = \left\{ \sup_{0 \le r2^{-2m} \le 2a_m} |s_m(r)
- r2^{-2m}| \le h_{a,m} \right\}, \] if $r = \lfloor (\qv{t}
\wedge a_m + h_{a,m})2^{2m} \rfloor +1$, then $s_m(r) > \qv{t}
\wedge a_m$, so $s_m(r)$ is not included in $N_m(t\wedge
T_{a_m})$. Observe here that $a_m+h_{a,m}+2^{-2m} \le 2a_m$ if $m$
is large enough, $m \ge m_4(C)$, where we also suppose that
$m_4(C) \ge m_3(C)$ and $m_3(C)$ is defined by Lemma
\ref{le:equid}. This explains why the sup is taken for $r2^{-2m}
\le 2 a_m$ in the definition of $A_{a,m}$. Similarly on $A_{a,m}$,
if $r = \lfloor (\qv{t} \wedge a_m - h_{a,m})2^{2m} \rfloor$, then
$s_m(r) \le \qv{t} \wedge a_m$, so $s_m(r)$ must be included in
$N_m(t\wedge T_{a_m})$. Hence
\begin{equation}
\qv{t} \wedge a_m  - h_{a,m} - 2^{-2m} \le N_m(t \wedge T_{a_m})
\le \qv{t} \wedge a_m + h_{a,m} + 2^{-2m} ,
\label{eq:qvardiff}
\end{equation}
for any $t \in [0, K]$ on $A_{a,m}$.

Now $6 \left(C 2a_m \log_*(2a_m) \right)^{\half} m^{\half} 2^{-m}
\le 11.1 \left(C a_m \log_*a_m \right)^{\half} m^{\half} 2^{-m}
= h_{a,m}$, since $\log_*(2 a_m) \le (1 + \log 2) \log_* a_m$. Hence
it follows by Lemma \ref{le:equid} that
\[
\PP \lb A_{a,m}^c \rb \le 4(2a_m 2^{2m})^{1-C}
\le 3(a_m 2^{2m})^{1-C},
\]
when $C\ge 3/2$ and $m \ge m_4(C)$. Noticing that $0.9 \left(C a_m
\log_*a_m \right)^{\half} m^{\half} 2^{-m}$ $> 2^{-2m}$ for any $m
\ge 1$, this and (\ref{eq:qvardiff}) prove (a).

Part (b) follows from (a), the inequality
\begin{eqnarray}
| \qv{t} - N_m(t) | &\le& | \qv{t} - \qv{t} \wedge a_m| \nonumber \\
&+& | \qv{t} \wedge a_m - N_m(t \wedge T_{a_m}) | \nonumber \\
&+& | N_m(t \wedge T_{a_m}) - N_m(t) | , \label{eq:triang}
\end{eqnarray}
and from the following simple relationships between events:
\[ \{ \qv{t} \wedge a_m \ne \qv{t} \} = \{ \qv{t} > a_m \} \]
and
\[
\{ N_m(t \wedge T_{a_m}) \ne N_m(t) \} = \{ t > T_{a_m} \}
= \{ \qv{t} > a_m \} ,
\]
cf. (\ref{eq:tchange2}).
\end{proof}

\begin{rem} \label{re:rem}
When the quadratic variation $\qv{t}$ is almost
surely bounded above by a finite valued function $g(t)\ge \max\{t,e\}$ for each $t >
0$, statements (a) and (b) of Lemma \ref{le:approx} simplify as
\begin{eqnarray*}
\PP \lb \sup_{0 \le t \le K} \left| \qv{t} - N_m(t) \right|
\ge 12(C g(K) \log g(K))^{\half} m^{\half} 2^{-m} \rb \\
\le 3(g(K) 2^{2m})^{1-C}
\end{eqnarray*}
for any $K>0$, $C\ge 3/2$ and $m\ge m_4(C)$.
\end{rem}

The statement of the next theorem corresponds to the main result
in Karand\-ikar \cite{Kar83}, though the method applied is
different and here we give a rate of convergence as well.

\begin{thm} \label{th:qvar}
Using the same notations as in Lemma \ref{le:qvar} and taking a sequence $(c_m)$ increasing to $\infty$ arbitrary slowly, we have
\[
\sup_{0\le t\le K} \left|\qv{t}-N_m(t)\right| < c_m m^{\half}
2^{-m} \qquad \mbox{a.s.} \qquad (m \to \infty) .
\]
Under the condition of Remark \ref{re:rem}, we also have
\[
\sup_{0\le t\le K} \left|\qv{t}-N_m(t)\right| < g(K)^{\half} (\log g(K))^{\half}
\qquad \mbox{a.s.} \qquad (K \to \infty)
\]
for any $m$ large enough, $m \ge m_4(3)$.
\end{thm}

\begin{proof}
To show the first statement take e.g. $C=3/2$ and $a_m = c_m$
in Lemma \ref{le:qvar} (a). Consider the inequality (\ref{eq:triang}).
Since $\qv{K}$ is finite-valued and $c_m \to \infty$, if $m$ is large
enough, depending on $\omega$, $\qv{K} < a_m$ holds and then $t <
T_{a_m}$ holds as well by (\ref{eq:tchange1}). These remarks show
that the first and the third terms on the right hand side of
inequality (\ref{eq:triang}) are zero if $m$ is large enough.
Further, statement of Lemma \ref{le:qvar} (a) can be applied to the
second term.  This, with the Borel--Cantelli lemma, proves the
first statement of the theorem.

The second statement of theorem follows similarly from Lemma \ref{le:qvar}
(a) by the Borel--Cantelli lemma, taking $C=3$ and $a_m = g(K)$.
\end{proof}

Now we are ready to discuss the strong approximation of continuous
local martingales by time-changed random walks.

\begin{lem} \label{le:approx}
Let M be a continuous local martingale vanishing at 0, let $\qv{}$ be
the quadratic variation and $T$ be its quasi-inverse (\ref{eq:T}).
Denote by $B_m$ the sequence of shrunken RWs embedded into $M$ by
Lemma \ref{le:SkorM}. Fix $K > 0$ and take a sequence $a_m
= O(m^{-7-\epsilon} 2^{2m}) K$ with some $\epsilon > 0$,
where $a_m \ge K \lor 1$ for any $m \ge 1$.

(a) Then for any $C\ge 3/2$ and $m\ge m_3(C)$ we have
\begin{eqnarray*}
\PP \lb \sup_{0\le t\le K}|M(t \wedge T_{a_m}) -
B_m(\qv{t} \wedge a_m)| \ge  a_m^{\quart} (\log_*a_m)^{\frac34} m
2^{-\frac{m}{2}} \rb \\
\le 10(a_m 2^{2m})^{1-C} .
\end{eqnarray*}

(b) Under the tail-condition (\ref{eq:tail}),
for any $C\ge 3/2$ and $m\ge m_3(C)$ it follows that
\begin{eqnarray*}
\PP \lb \sup_{0 \le t \le K} \left| M(t) - B_m(\qv{t}) \right|
\ge a_m^{\quart} (\log_*a_m)^{\frac34} m 2^{-\frac{m}{2}}  \rb \\
\le 10(a_m 2^{2m})^{1-C} + D(K) m^{-1-\epsilon} .
\end{eqnarray*}
\end{lem}

\begin{proof}
First, take the DDS Wiener process $W(s)$ obtained from $M(t)$ by the
time-change $T_s$ as described by Theorems \ref{DDS1} and \ref{DDS2}.
Since below we are going to use $W(s)$ and also the time change $T_s$
only for arguments $s \le \qv{\infty}$, we can always assume that
$W(s)=M(T_s)$ and $M(t)=W(\qv{t})$,
irrespective of the fact whether $\qv{\infty} = \infty$ or not.
Second, define the nested sequence of shrunken RWs $B_m(s)$ by
Lemma \ref{le:SkorM}. Then a quasi-inverse time-change $\qv{t}$ is applied
to $B_m(s)$ that gives $B_m(\qv{t})$ which will be the
sequence of time-changed shrunken RWs approximating $M(t)$.

Since $T_s$ may have jumps, we get that
\begin{eqnarray}
\sup_{0\le t \le K} \left| M(t) - B_m(\qv{t}) \right|
&\ge& \sup_{0\le s \le \qv{K}} \left| M(T_s) - B_m(\qv{T_s}) \right|
\nonumber \\
&=& \sup_{0\le s \le \qv{K}} \left| W(s) - B_m(s) \right| .
\label{eq:mtchange}
\end{eqnarray}
Recalling however that the intervals of constancy are the same
for $M(t)$ and for $\qv{t}$ \cite[IV (1.13), p.125] {Yor99},
there is in fact equality in (\ref{eq:mtchange}). To go on, we
need a truncation using the sequence $a_m$, since the
quadratic variation $\qv{t}$ is not a bounded
random variable in general. Then (\ref{eq:mtchange}) (with equality
as explained above) and (\ref{eq:tchange2}) imply
\begin{eqnarray*}
\lefteqn{\sup_{0\le t \le K} \left|M(t\wedge T_{a_m})
- B_m(\qv{t} \wedge a_m) \right|} \\
&=& \sup_{0\le s \le \qv{K}} \left|M(T_s \wedge
T_{a_m}) - B_m(\qv{T_s} \wedge a_m) \right| \\
&=& \sup_{0\le s \le a_m \wedge \qv{K}} \left| W(s) - B_m(s) \right| \\
&\le& \sup_{0\le s \le a_m} \left| W(s) - B_m(s) \right| .
\end{eqnarray*}
Hence by Theorem \ref{th:Wienerm}, with $m \ge m_3(C)$,
\begin{eqnarray*}
\lefteqn{\PP \lb \sup_{0\le t\le K} |M(t \wedge T_{a_m})
- B_m(\qv{t} \wedge a_m)| \ge a_m^{\quart} (\log_*a_m)^{\frac34} m
2^{-\frac{m}{2}} \rb} \\
&\le& \PP \lb \sup_{0\le s \le a_m} \left| W(s) - B_m(s) \right|
\ge a_m^{\quart} (\log_*a_m)^{\frac34} m 2^{-\frac{m}{2}} \rb \\
&\le& 10(a_m2^{2m})^{1-C} .
\end{eqnarray*}
This proves (a).

To show (b) it is enough to consider the inequality
\begin{eqnarray*}
&& \sup_{0\le t\le K} \left| M(t) - B_m(\qv{t}) \right| \\
&\le& \sup_{0\le t\le K} \left| M(t) - M(t \wedge T_{a_m}) \right|
+ \sup_{0\le t\le K} \left| M(t \wedge T_{a_m})
- B_m(\qv{t} \wedge a_m) \right| \\
&& + \sup_{0\le t\le K} \left| B_m(\qv{t} \wedge a_m)
- B_m(\qv{t}) \right| .
\end{eqnarray*}
From this point the proof is similar to the proof of Lemma
\ref{le:qvar} (b).
\end{proof}

Kiefer \cite{Kie69} proved in the Brownian case $M = W$ that using
Skorohod embedding one cannot embed a standardized
RW into $W$ with convergence rate better
than $O(1) n^{-\quart} (\log n)^{\half} (\log \log n)^{\quart}$,
where $n$ is the number of points used in the
approximation. Since the next theorem gives a rate of convergence
$O(1) n^{-\quart} \log n$ (the number of points used is
$n = K 2^{2m}$), this rate is close to the best we can have with a
Skorohod-type embedding. The same remark is valid for Theorem
\ref{th:approxNm} below.

\begin{thm} \label{th:approx}
Applying the same notations as in Lemma \ref{le:approx} and taking a sequence $(c_m)$ increasing to $\infty$ arbitrary slowly, we have
\[
\sup_{0\le t\le K} \left| M(t) - B_m(\qv{t}) \right| < c_m
m 2^{-\frac{m}{2}} \qquad \mbox{a.s.} \qquad (m \to \infty).
\]
Under the condition of Remark \ref{re:rem}, we also have
\[
\sup_{0\le t\le K} \left| M(t) - B_m(\qv{t}) \right| < g(K)^{\frac14}
(\log g(K))^{\frac34} \qquad \mbox{a.s.} \qquad (K \to \infty)
\]
for any $m$ large enough, $m \ge m_3(3)$.
\end{thm}

\begin{proof}
The statements follow from Lemma \ref{le:approx} in a similar way
as Theorem \ref{th:qvar} followed from Lemma \ref{le:qvar}.
\end{proof}

We mention that when $M$ is a continuous local martingale vanishing
at 0 and there is a
deterministic function $f$ on $\RR_+$ such that $\qv{t}=f(t)$ a.s.,
then it follows that $M$ is Gaussian and has independent increments,
see \cite[V (1.14), p.186]{Yor99}.  Conversely, if $M$ is a
continuous Gaussian martingale, then $\qv{t}=f(t)$ a.s., see
\cite[IV (1.35), p.133]{Yor99}.
In this case the ``twist and shrink'' construction of Brownian motion
described in Section 2 can be extended to a construction of $M(t)$
(or a simulation algorithm in practice). Namely, we have
\[
\left| M(t) - \tilde{B}_m\left(f(t)\right) \right|
\le O(1) m 2^{-\frac{m}{2}} \qquad \mbox{a.s.} \qquad (m \to \infty).
\]
Here $\tilde{B}_m(t) = 2^{-m} \tilde{S}_m(t2^{2m})$ ($m \ge 0$)
denotes the nested sequence of the RW construction described in
Section 2.

Combining the previous results one can replace $\qv{t}$ by the
discrete quadratic variation $N_m(t)$ when approximating $M(t)$ by
time-changed shrunken RWs.

\begin{lem} \label{le:approxNm}
Let M be a continuous local martingale vanishing at 0, let $\qv{}$ be
the quadratic variation, $T$ be its quasi-inverse (\ref{eq:T}) and
$N_m$ be the discrete quadratic variation defined by (\ref{eq:N_m}).
Denote by $B_m$ the sequence of shrunken RWs embedded into $M$ by
Lemma \ref{le:SkorM}. Fix $K > 0$ and take a sequence $a_m
= O(m^{-7-\epsilon} 2^{2m}) K$ with some $\epsilon > 0$, where
$a_m \ge K \lor 1$ for any $m \ge 1$.

(a)  Then for any $C\ge 3/2$ and $m\ge m_5(C)$ we have
\begin{eqnarray*}
\PP \lb \sup_{0\le t\le K}|M(t \wedge T_{a_m}) -
B_m(N_m(t \wedge T_{a_m})| \ge 2 a_m^{\quart} (\log_*a_m)^{\frac34} m
2^{-\frac{m}{2}} \rb \\
\le 14(a_m 2^{2m})^{1-C} .
\end{eqnarray*}

(b) Under the tail-condition (\ref{eq:tail}), for any
$C\ge 3/2$ and $m\ge m_5(C)$ it follows that
\begin{eqnarray*}
\PP \lb \sup_{0 \le t \le K} \left| M(t)
- B_m(N_m(t)) \right| \ge 2 a_m^{\quart} (\log_*a_m)^{\frac34} m
2^{-\frac{m}{2}} \rb \\
\le 14(a_m 2^{2m})^{1-C} + D(K) m^{-1-\epsilon} .
\end{eqnarray*}
\end{lem}

\begin{proof}
For proving (a) we use the triangle inequality
\begin{eqnarray}
\lefteqn{\sup_{0\le t\le K} |M(t \wedge T_{a_m}) - B_m(N_m(t
\wedge T_{a_m})|} \nonumber \\
&\le& \sup_{0\le t\le K} |M(t \wedge T_{a_m}) - B_m(\qv{t} \wedge a_m)|
\nonumber \\
&&+ \sup_{0\le t\le K} |B_m(\qv{t} \wedge a_m) - B_m(N_m(t \wedge
T_{a_m})| . \label{eq:triang2}
\end{eqnarray}
Since the first term on the right hand side can be estimated by
Theorem \ref{th:approx}, we have to consider the second term. For
$m \ge 1$ introduce the abbreviation
\[
L_{a,m} = 13 (C a_m \log_* a_m)^{\half} m^{\half} 2^{-m}
\ge 12 (C a_m \log_* a_m)^{\half} m^{\half} 2^{-m} + 2^{-2m} .
\]
By Theorem \ref{th:qvar} (a), with $C \ge 3/2$ and $m \ge m_4(C)$
it follows that
\begin{eqnarray*}
\lefteqn{\sup_{0\le t\le K} |B_m(\qv{t} \wedge a_m)
- B_m(N_m(t \wedge T_{a_m})|} \\
&\le& \sup_{0\le t\le K} \left| B_m \left( \lfloor (\qv{t} \wedge a_m)
2^{2m} \rfloor 2^{-2m} \right) - B_m(N_m(t \wedge T_{a_m}) \right|
+ 2^{-m} \\
&\le& \sup_{0 \le k2^{-2m} \le \qv{K} \wedge a_m} \quad \sup_{|r-k|
2^{-2m} \le L_{a,m}} |B_m(k2^{-2m}) - B_m(r2^{-2m})| + 2^{-m} \\
&\le& \sup_{0 \le k2^{-2m} \le a_m} \quad \sup_{0 \le r2^{-2m} \le
L_{a,m}} |B_m^{(k)}(r2^{-2m})| + 2^{-m} ,
\end{eqnarray*}
except for an event of probability $\le 3(a_m2^{2m})^{1-C}$,
since the difference of a shrunken RW at two dyadic points equals the
value of some shrunken RW $B_m^{(k)}$ at a dyadic point.

Then we can apply estimate (\ref{eq:ldp2}) with some $C'>1$ for the last
expression:
\begin{eqnarray*}
&& \PP \lb \sup_{0 \le k2^{-2m} \le a_m} \sup_{0 \le r2^{-2m}
\le L_{a,m}} |B_m^{(k)}(r2^{-2m})| \right. \\
&& \left. \ge \left(2 C' \log N \:
\sup_{k, r} \Var(B_m^{(k)}(r2^{-2m}))\right)^{\half} \rb
\le 2 N^{1-C'} ,
\end{eqnarray*}
where $N = \lfloor a_m 2^{2m} \rfloor \lfloor L_{a,m} 2^{2m} \rfloor$
and $\sup_{k,r} \Var\left(B_m^{(k)}(r2^{-2m})\right) \le L_{a,m}$.
Choose here $C'$ so that $1-C' = \frac23 (1-C)$. Then a
simple computation shows that $2 N^{1-C'} \le (a_m 2^{2m})^{1-C}$,
also $\log N \le 8 m \log_* C \log_* a_m$, and
\[\left(2 C' \log N \sup_{k,r} \Var(B_m^{(k)}(r2^{-2m}))
\right)^{\half} + 2^{-m} \le a_m^{\quart} (\log_*a_m)^{\frac34} m
2^{-\frac{m}{2}} \]
if $m \ge m_5(C) \ge m_4(C)$. This argument and Theorem
\ref{th:approx}(a) applied to (\ref{eq:triang2}) give (a).

Statement (b) again follows from (a) in a similar way as
in Lemma \ref{le:qvar}.
\end{proof}

\begin{thm} \label{th:approxNm}
Using the same notations as in Lemma \ref{le:approxNm} and taking a sequence $(c_m)$ increasing to $\infty$ arbitrary slowly, we have
\[
\sup_{0\le t\le K} \left| M(t) - B_m(N_m(t)) \right| < c_m
m 2^{-\frac{m}{2}} \qquad \mbox{a.s.} \qquad (m \to \infty).
\]
Under the condition of Remark \ref{re:rem}, we also have
\[
\sup_{0\le t\le K} \left| M(t) - B_m(N_m(t)) \right| <
g(K)^{\frac14} (\log g(K))^{\frac34} \qquad \mbox{a.s.} \qquad (K \to \infty)
\]
for any $m$ large enough, $m \ge m_5(3)$.
\end{thm}

\begin{proof}
The statements follow from Lemma \ref{le:approxNm} in a similar way again as Theorem \ref{th:qvar} followed from Lemma \ref{le:qvar}.
\end{proof}

\section{Independence of the DDS BM and $\qv{}$}

It is important both from theoretical and practical (e.g.
simulation) points of view that the shrunken RW $B_m$ and the
corresponding discrete quadratic variation process $N_m$ be
independent when approximating $M$ as in Theorem
\ref{th:approxNm}. This leads to the question of independence of
the DDS Brownian motion $W$ and quadratic variation $\qv{}$ in the
case of a continuous local martingale $M$. For, by Lemma
\ref{le:SkorM}, $B_m$ depends only on $W$ and, by (\ref{eq:N_m}),
$N_m$ is determined by $\qv{}$ alone.  Conversely, if the
processes $B_m$ and $N_m$ are independent for any $m$ large
enough, then so are $W$ and $\qv{}$ too by Theorems
\ref{th:Wienerm} and \ref{th:qvar}. It will turn out from the next
theorem that the basic notion in this respect is the symmetry of
the increments of $M$ given the past. Thus we will say that a
stochastic process $M(t)$ $(t \ge 0)$ is {\em symmetrically
evolving} (or has symmetric increments given the past) if for any
positive integer $n$, reals $0 \le s < t_1 < \dots <t_n$ and Borel
sets of the line $U_1, \dots U_n$ we have
\begin{equation}
\PP \lb \Gamma \mid \mathcal{F}_s^0 \rb
= \PP \lb  \Gamma^- \mid \mathcal{F}_s^0 \rb  \label{eq:indep1} ,
\end{equation}
where $\Gamma = \lb M(t_1) - M(s) \in U_1, \dots , M(t_n) - M(s)
\in U_n \rb$, $\Gamma^-$ is the same, but each $U_j$  replaced by
$-U_j$, and $\mathcal{F}_s^0 = \sigma(M(u), 0 \le u \le s)$ is the
filtration generated by the past of $M$.  If $M(t)$ has finite
expectation for any $t \ge 0$, then this condition expresses a
very strong martingale property.

Condition (\ref{eq:indep1}) is clearly equivalent to
the following one: for arbitrary positive integers $n$, $j$, reals
$0 \le s_j < \dots < s_1 \le s < t_1 <
\dots t_n$ and Borel-sets $V_1, \dots , V_j, U_1, \dots , U_n$ one has
\begin{equation}
\PP \lb \Gamma \cap \Lambda \rb = \PP \lb \Gamma^- \cap  \Lambda \rb ,
\label{eq:indep2}
\end{equation}
where $\Gamma$ and $\Gamma^-$ are defined above and $\Lambda = \lb
M(s_1) \in V_1, \dots , M(s_j) \in V_j \rb$.

Our Theorem \ref{th:indep} below is basically a reformulation of
Ocone's Theorem A of \cite{Oco93}. There it is shown that a
continuous local martingale $M$ is conditionally (w.r.t. to the
sigma algebra generated by $\qv{}$) Gaussian martingale if and
only if it is  $J$-invariant. Here $J$-invariance means that $M$
and $\int_0^t \alpha \di M$ have the same law for any predictable
process $\alpha$ with range in $\lb -1, 1 \rb$. In fact, it is
proved there too that $J$-invariance is equivalent to
$H$-invariance which means that it is enough to consider
deterministic integrands of the form $\alpha^{(r)}(t) = \II_{[0,
r]}(t) - \II_{(r, \infty)}(t)$. Moreover, Theorem B there extends
the above result to c\`adl\`ag local martingales with symmetric
jumps.

Dubins, \'Emery and Yor in \cite{DEY93} and Vostrikova and Yor in
\cite{VoY00} gave shorter proofs with additional equivalent
conditions in the case when $M$ is a continuous martingale. In
these references the equivalent condition of the independence of
the DDS BM and $\qv{}$ explicitly appears. Besides, in
\cite{DEY93}, the conjecture that a continuous martingale $M$ has
the same law as its L\'evy transform $\hat{M} = \int \mbox{sgn}(M)
\di M$ if and only if its DDS BM and $\qv{}$ are independent is
proved to be equivalent to the conjecture that the L\'evy
transform is ergodic. Below we give a new, long, but elementary
proof for any continuous local martingale $M$  that the DDS BM and
$\qv{}$ are independent if and only if $M$ is symmetrically
evolving.

\begin{thm}\label{th:indep}
(a) If the Wiener process $W(t)$ $(t \ge 0)$ and the non-decreasing,
vanishing at 0, continuous stochastic process $C(t)$ $(t \ge 0)$ are
independent, then $M(t) = W(C(t))$ is a symmetrically evolving
continuous local martingale vanishing at 0, with quadratic variation
$C$.

(b) Conversely, if $M$ is a symmetrically evolving continuous local
martingale, then its DDS Brownian motion $W$ and its quadratic variation
$\qv{}$ are independent processes.
\end{thm}

\begin{proof}
To prove (a) suppose that $W$ and $C$ are independent. By \cite[V (1.5),
p. 181]{Yor99}, $M(t)=W(C(t))$ is a continuous local martingale. For
simplicity, we will use only three sets in showing that $M$ is
symmetrically evolving, i.e. equation (\ref{eq:indep2}) holds, the
generalization being straightforward:
\begin{multline*}
\PP \lb M(s_1) \in V_1, M(t_1)-M(s) \in U_1,
M(t_2)-M(s) \in U_2 \rb \\
= \int \PP \lb W(x_1) \in V_1 \rb \\
\int_{U_1} \PP \lb W(y_2)
-W(y_1) \in U_2-u \rb \PP \lb W(y_1)-W(y_0) \in \di u \rb \\
\times \PP \lb C(s_1) \in \di x_1, C(s) \in \di y_0, C(t_1) \in \di
y_1, C(t_2) \in \di y_2 \rb \\
= \int \PP \lb W(x_1) \in V_1 \rb \\
\int_{U_1} \PP \lb W(y_1)
-W(y_2) \in U_2-u \rb \PP \lb W(y_0)-W(y_1) \in \di u \rb \\
\times \PP \lb C(s_1) \in \di x_1, C(s) \in \di y_0, C(t_1) \in
\di y_1, C(t_2) \in \di y_2 \rb \\
= \PP \lb M(s_1) \in V_1, M(s)-M(t_1) \in U_1,
M(s)-M(t_2) \in U_2 \rb ,
\end{multline*}
using the independence of $B$ and $C$ on one hand and the
symmetry and independence of the increments of Brownian motion on
the other hand.

For proving (b) we want to show that the sequences $\tau_m(k)$
$(k=1,2, \dots )$ and $M(\tau_m(j))-M(\tau_m(j-1))$ $(j=1,2, \dots
)$ are independent. Since $N_m(t)$ depends only on the number of
stopping times $\tau_m(k) \le t$, cf. (\ref{eq:N_m}), while the
shrunken random walk $B_m$ is determined by the steps
$2^{-m}X_m(j) = M(\tau_m(j))-M(\tau_m(j-1))$, cf. Lemma
\ref{le:SkorM}, this would imply their independence and so the
independence of $W$ and $\qv{}$ too by Theorems \ref{th:Wienerm}
and \ref{th:qvar}. For this it is enough to show that with
arbitrary integers $m \ge 0$, $n \ge 1$, $0 \le k < n$, reals
$t_1, \dots , t_n$ and $\delta_1=\pm 2^{-m}, \dots , \delta_n=\pm
2^{-m}$ (we fix these parameters for the remaining part of the
proof) one has
\begin{equation}
\PP \lb A \cap B_{\le k} \cap B_{>k} \rb
= \PP \lb A \cap B_{\le k} \cap B_{>k}^- \rb,
\label{eq:indepd}
\end{equation}
where $A_{\le k} = \bigcap_{r=1}^k \lb \tau_m(r) \le t_r \rb$,
$A_{> k}$ is similar, but with $r = k+1, \dots n$, $A = A_{\le k}
\cap A_{>k}$, $B_{\le k} = \bigcap_{r=1}^k \lb  M(\tau_m(r)) -
M(\tau_m(r-1)) = \delta_r \rb$, $B_{>k}$ is similar, but with $r =
k+1, \dots n$, $B = B_{\le k} \cap B_{>k}$ , and finally
$B_{>k}^-$ is the same as $B_{>k}$, but each $\delta_j$ is
replaced by $-\delta_j$. For, if one can reflect all $\delta_j$s
for $k < j \le n$ without changing the probability, then one has
the same probability with arbitrary changed signs of $\delta_j$s
too, since any such change can be reduced to a finite sequence of
reflections of the above type. Let $B^*$ be similar to $B$, but
with arbitrarily changed signs of $\delta_j$s. Then, as we said,
(\ref{eq:indepd}) implies that $\PP \lb A \cap B \rb =\PP \lb A
\cap B^* \rb$. Since $\PP \lb B \rb = \PP \lb B^* \rb$ by Lemma
\ref{le:SkorM}, the desired independence follows.

We will prove (\ref{eq:indepd}) in several steps.

\emph{Step 1.}
In condition (\ref{eq:indep1}) one can replace $s$ by an
arbitrary stopping time $\sigma$ adapted to the filtration
$(\mathcal{F}_s^0)$: for any $u_j \ge 0$ $(1 \le j \le N)$,
\begin{equation}
\PP \lb F \mid \mathcal{F}_{\sigma}^0 \rb = \PP \lb F^- \mid
\mathcal{F}_{\sigma}^0 \rb , \label{eq:sigma}
\end{equation}
where
\begin{equation}
F = \bigcap_{j=1}^N \lb M(u_j+\sigma) - M(\sigma) \in U_j \rb ,
\label{eq:Fsigma}
\end{equation}
and $F^-$ is the same, but each $U_j$ replaced by $-U_j$. This is
somewhat similar to the optional stopping theorem, see \cite[II
(3.2), p.69]{Yor99}. Indeed, for discrete valued stopping times
$\sigma$ the statement is obvious, since then
\[
\PP \lb F \mid \mathcal{F}_{\sigma}^0 \rb
= \sum_{s_r} \II\lb \sigma = s_r \rb
\PP \lb F \mid \mathcal{F}_{s_r}^0 \rb ,
\]
where $\lb s_r \rb$ denotes the range of $\sigma$, including possibly
$\infty$, and $\II\lb S \rb$ denotes the indicator of the set $S$.
For every stopping time $\sigma$ there
exists a decreasing sequence of discrete valued stopping times $\sigma_i$
almost surely converging to $\sigma$. Let us denote the events defined according
to (\ref{eq:Fsigma}) for $\sigma_i$ by $F_i$ and $F_i^-$, respectively.
Further, denote the operators projecting
$L^2(\Omega)$ onto its subspace of random variables measurable
w.r.t. $\mathcal{F}_{\sigma_i}^0$ and $\mathcal{F}_{\sigma}^0$ by $P_i$
and $P$, respectively. Then
\begin{eqnarray*}
&& \| \PP \lb F_i \mid \mathcal{F}_{\sigma_i}^0 \rb - \PP \lb F \mid
\mathcal{F}_{\sigma}^0 \rb \|_2
= \| P_i \II\lb F_i \rb - P \II \lb F \rb \|_2 \\
&\le& \| P_i \left( \II\lb F_i \rb - \II \lb F \rb \right) \|_2
+ \| P_i \II \lb F \rb - P \II \lb F \rb \|_2 \\
&\le& \| \II\lb F_i \rb - \II \lb F \rb \|_2
+ \| \EE(\II \lb F \rb \mid \mathcal{F}_{\sigma_i}^0 ) - \EE(\II \lb
F \rb \mid \mathcal{F}_{\sigma}^0 ) \|_2 ,
\end{eqnarray*}
which goes to 0 as $i \to \infty$. Here we used that $\EE(\II \lb F
\rb \mid \mathcal{F}_{\sigma_i}^0 )$ is a bounded, reversed-time
martingale converging to $\EE(\II \lb F \rb \mid \mathcal{F}_
{\sigma}^0 )$. Hence for any $\epsilon > 0$,
\begin{eqnarray*}
\| \PP \lb F \mid \mathcal{F}_{\sigma}^0 \rb - \PP \lb F^- \mid
\mathcal{F}_{\sigma}^0 \rb \|_2  &\le& \| \PP \lb F \mid
\mathcal{F}_{\sigma}^0 \rb
- \PP \lb F_i \mid \mathcal{F}_{\sigma_i}^0 \rb \|_2 \\
&& + \| \PP \lb F^- \mid
\mathcal{F}_{\sigma}^0 \rb - \PP \lb F_i^- \mid \mathcal{F}_{\sigma_i}^0
\rb \|_2 < \epsilon ,
\end{eqnarray*}
if $i$ is large enough. This shows that the left hand side of the
inequality is zero, so (\ref{eq:sigma}) holds.

\emph{Step 2.}
Then for arbitrary reals $0 \le u_j < v_j$ and Borel sets $U_j$ $(1 \le j
\le N)$ we have
\[
\PP \lb G \mid \mathcal{F}_\sigma^0 \rb
=  \PP \lb G^- \mid \mathcal{F}_\sigma^0 \rb ,
\]
where
\begin{equation}
G = \bigcap_{j=1}^N \lb M(v_j+\sigma) - M(u_j+\sigma) \in U_j \rb ,
\label{eq:Guv}
\end{equation}
and $G^-$ is the same, but each $U_j$ is replaced by $-U_j$. For
simplicity we prove this only for two factors, the general case
being similar:
\begin{multline*}
\PP \lb M(v_1+\sigma) - M(u_1+\sigma) \in U_1,
M(v_2+\sigma) - M(u_2+\sigma) \in U_2 \mid \mathcal{F}_\sigma^0 \rb \\
= \int \II \{x_1-x_2 \in U_1\} \: \II \{x_3-x_4 \in U_2\} \\
\times \PP \lb M(v_1+\sigma)-M(\sigma) \in \di x_1, M(u_1+\sigma)-M(\sigma) \in \di x_2, \right. \\
\left. M(v_2+\sigma)-M(\sigma) \in \di x_3,
M(u_2+\sigma)-M(\sigma) \in \di x_4 \mid \mathcal{F}_\sigma^0 \rb \\
= \int \II \{x_1-x_2 \in U_1\} \: \II \{x_3-x_4 \in U_2\}  \\
\times \PP \lb M(\sigma)-M(v_1+\sigma) \in \di x_1, M(\sigma)-M(u_1+\sigma) \in \di x_2, \right. \\
\left. M(\sigma)-M(v_2+\sigma) \in \di x_3,
M(\sigma)-M(u_2+\sigma) \in \di x_4 \mid \mathcal{F}_\sigma^0 \rb \\
= \PP \lb M(u_1+\sigma) - M(v_1+\sigma) \in U_1,
M(u_2+\sigma) - M(v_2+\sigma) \in U_2 \mid \mathcal{F}_\sigma^0 \rb .
\end{multline*}

\emph{Step 3.}
Let $\Delta \tau_m(i) = \tau_m(i) - \tau_m(i-1)$ and $a \in [0, \infty)$.
Consider the event
\begin{eqnarray}
S(a) &=& \lb \Delta \tau_m(k+1) \ge a \rb \nonumber \\
&=& \lb \inf \lb u>0 : |M(u+\tau_m(k))-M(\tau_m(k))| \ge 2^{-m} \rb > a \rb
\nonumber \\
&=& \lb \sup_{0 < u \le a} \lb |M(u+\tau_m(k))-M(\tau_m(k))| \rb < 2^{-m}
\rb \nonumber \\
&=& \bigcap_{0<u \le a} \lb |M(u+\tau_m(k))-M(\tau_m(k))| < 2^{-m} \rb .
\label{eq:Sa}
\end{eqnarray}
Introduce the set of dyadic numbers $D_l = \lb r2^{-l}: r \in \ZZ \rb$
$(l \ge 0)$, and the events
\begin{equation}
S_{j,l}(a) = \bigcap_{q \in D_l, 0 < q \le a} \lb |M(q+\tau_m(k))
- M(\tau_m(k))| \le 2^{-m} - 2^{-j} \rb  .
\label{eq:Sjla}
\end{equation}
when $j \ge m$. For $j$ fixed,
$\left(S_{j,l}\right)_{l=0}^{\infty}$ is a decreasing sequence of
events. Take
\[
S_{j}(a) = \bigcap_{l=0}^{\infty} S_{j,l}(a) .
\]
Since $S_{j,l}(a)$ is increasing with growing $j$, so is
$S_{j}(a)$. Put
\[
S^*(a) = \bigcup_{j=m}^{\infty} S_{j}(a)
= \bigcup_{j=m}^{\infty} \bigcap_{l=0}^{\infty} S_{j,l}(a).
\]
We want to show that $S^*(a) = S(a)$, where the latter is
defined by (\ref{eq:Sa}).

First fix an $\omega \in S(a)$. (We suppress $\omega$ in the
notations below.) Then with this $\omega$,
\[
\sup_{0 < u \le a} \lb |M(u+\tau_m(k))-M(\tau_m(k))| < 2^{-m} \rb
=: s < 2^{-m} .
\]
If $j>m$ is such that $2^{-j} < 2^{-m} - s$, then $\omega \in
S_{j,l}(a)$ for any $l \ge 0$. So $\omega \in S_{j}(a)$ for each
$j$ large enough, consequently, $\omega \in S^*(a)$.

Second, fix an $\omega \notin S(a)$. Then there exists a real $u_0
\le a$ (depending on $\omega$) so that $|M(u_0+\tau_m(k))-M(\tau_m(k))|
= 2^{-m}$. Since the path of $M$ is a continuous function, for any $j
\ge m$ there exists an $l \ge 0$ and $q \in D_l, 0 < q \le a$ such that
$|M(q  +\tau_m(k))-M(\tau_m(k))| > 2^{-m} - 2^{-j}$. That is, $\omega
\notin S_{j}(a)$ if $j \ge m$, thus $\omega \notin S^*(a)$.

In other words, we have proved above that
\begin{eqnarray*}
&& \lb \Delta \tau_m(k+1) > a \rb = S(a) = \lim_{j
\to \infty} \lim_{l \to \infty}  S_{j,l}(a) \\
&=& \lim_{j \to \infty} \lim_{l \to \infty} \bigcap_{q \in D_l, 0 < q \le a}
\lb |M(q  +\tau_m(k))-M(\tau_m(k))| \le 2^{-m} - 2^{-j} \rb .
\end{eqnarray*}
Consequently, any event $\lb \Delta \tau_m(k+1) > a \rb = S(a)$ can
be written in terms of monotonic sequences of intersections of
finitely many events of the form
\[
\lb |M(q +\tau_m(k)) - M(\tau_m(k))| \le c \rb \qquad (c \ge 0) .
\]
Moreover, such an approximation can be applied to $\lb \Delta \tau_m(k+1)
\in (a, b] \rb = S(a) \setminus S(b)$ as well, with any
$0 \le a < b$.

\emph{Step 4.}
First, Steps 2 and 3 imply that for any $a \ge 0$,
\[
\PP \lb G \cap S_{j,l}(a) \mid \mathcal{F}_{\tau_m(k)}^0 \rb
= \PP \lb G^- \cap S_{j,l}(a) \mid \mathcal{F}_{\tau_m(k)}^0 \rb ,
\]
because of the absolute value in definition (\ref{eq:Sjla}) of
$S_{j,l}(a)$. Throughout Step 4 $G$ and $G^-$ are defined according to
(\ref{eq:Guv}) with $\sigma = \tau_m(k)$, but otherwise with arbitrary
parameters, possibly different from case to case.
Then taking limit as $j \to \infty$ and $l \to \infty$ it follows
from Steps 2 and 3 that
\begin{eqnarray*}
&& \PP \lb G \cap \lb \Delta \tau_m(k+1) > a_{k+1} \rb \mid
\mathcal{F}_{\tau_m(k)}^0 \rb
\\
&=& \PP \lb G^- \cap \lb \Delta \tau_m(k+1) > a_{k+1} \rb \mid
\mathcal{F}_{\tau_m(k)}^0 \rb .
\end{eqnarray*}

We want to extend this symmetry property by induction over $i=k+1,
\dots , n+1$. Taking arbitrary reals $a_i \ge 0$ and integers $l
\ge 0$, $r_i > 0$, $(k+1 \le i \le n+1)$ define the events
\begin{eqnarray*}
&& B_i = \bigcap_{p=k+1}^i \lb M(\tau_m(p))-M(\tau_m(p-1)) = \delta_p
\rb , \\
&& H_i = \bigcap_{p=k+1}^i \lb \Delta \tau_m(p) > a_p \rb , \\
&& K_{i,l}(r) = \bigcap_{p=k+1}^i \lb \Delta
\tau_m(p) \in \left( (r_{p} - 1) 2^{-l}, r_p 2^{-l} \right] \rb \\
&& L_{i,l}(r) = \lb \delta_i
\left( M(r_i2^{-l} + \dots + r_{k+1}2^{-l} + \tau_m(k)) \right. \right. \\
&& \left. \left. \qquad \qquad - M(r_{i-1}2^{-l} + \dots + r_{k+1}2^{-l}
+ \tau_m(k)) \right) > 0 \rb ,
\end{eqnarray*}
and $B_i^-$, $L_{i,l}^-(r)$ similarly, but multiplying each $\delta_p$
by $(-1)$. Suppose that we have already proved that
\[
\PP \lb G \cap B_{i-1} \cap H_i \mid \mathcal{F}_{\tau_m(k)}^0 \rb
= \PP \lb G^- \cap B_{i-1}^- \cap H_i \mid \mathcal{F}_{\tau_m(k)}^0
\rb ,
\]
where $B_k = \Omega$. Define the following event as a generalization
of (\ref{eq:Sjla}):
\begin{eqnarray*}
S_{i,j,l}(a, r) &=& \bigcap_{q \in D_l, 0 < q
\le a} \lb \left| M( q + r_i 2^{-l} + \dots + r_{k+1} 2^{-l}
+ \tau_m(k)) \right. \right. \\
&& \qquad \qquad \left. \left. - M(r_{i-1}2^{-l}
+ r_{k+1}2^{-l} + \tau_m(k)) \right| \le 2^{-m} - 2^{-j} \rb ,
\end{eqnarray*}
where $j \ge m$. Then by the induction hypothesis we get that
\begin{multline*}
\PP \lb G \cap B_{i-1} \cap H_i \cap \bigcup_{r_{k+1}, \dots ,
r_i = 1}^{2^{2l}+1} K_{i,l}(r) \cap
L_{i,l}(r) \cap S_{i,j,l}(a_{i+1}, r)
\mid \mathcal{F}_{\tau_m(k)}^0 \rb \\
= \PP \lb G^- \cap B_{i-1}^- \cap H_i \cap \bigcup_{r_{k+1}, \dots ,
r_i = 1}^{2^{2l}+1} K_{i,l}(r) \cap
L_{i,l}^-(r) \cap S_{i,j,l}(a_{i+1}, r) \mid \mathcal{F}_{\tau_m(k)}^0
\rb ,
\end{multline*}
where we agree that when $r_{p} = 2^{2l}+1$, the interval $\left(
(r_{p} - 1) 2^{-l}, r_p 2^{-l} \right]$ in the definition of
$K_{i,l}(r)$ is replaced by $\left( (r_{p} - 1) 2^{-l}, \infty \right]
= ( 2^l, \infty ]$.
Notice here that the events in $K_{i,l}$ can be written in terms of a
difference of events appearing in $H_i$, while the events
in $L_{i,l}$ and $S_{i,j,l}$ are both of the type appearing in $G$,
though $S_{i,j,l}$ is not affected by reflections because of the
absolute values in its definition. Then taking limit as $j \to
\infty$ and $l \to \infty$ it follows that
\[
\PP \lb G \cap B_i \cap H_{i+1} \mid \mathcal{F}_{\tau_m(k)}^0 \rb
= \PP \lb G^- \cap B_i^- \cap H_{i+1} \mid \mathcal{F}_{\tau_m(k)}^0
\rb .
\]
This completes the induction.

Comparing the notations introduced in this step with the ones introduced above,
observe that $B_{>k} = B_n$ and $H_{>k} = H_n = H_{n+1}$, if
$a_{n+1}=0$. Thus one obtains that
\[
\PP \lb B_{>k} \cap H_{>k} \mid \mathcal{F}_{\tau_m(k)}^0 \rb
= \PP \lb B_{>k}^- \cap H_{>k} \mid \mathcal{F}_{\tau_m(k)}^0
\rb .
\]

\emph{Step 5.}
The result of Step 4 implies that
\begin{eqnarray*}
&& \PP \lb A_{> k} \cap B_{> k} \mid \mathcal{F}_{\tau_m(k)}^0 \rb \\
&=& \int \II \lb \tau_m(k)+x_{k+1} \le t_{k+1}, \dots ,
\tau_m(k)+x_{k+1}+\cdots +x_n \le t_n \rb \\
&& \times \PP \lb B_{> k} \cap \lb \Delta
\tau_m(k+1) \in \di x_{k+1}, \dots , \Delta \tau_m(n) \in \di x_n \rb
\mid \mathcal{F}_{\tau_m(k)}^0 \rb \\
&=& \int \II \lb \tau_m(k)+x_{k+1} \le t_{k+1}, \dots ,
\tau_m(k)+x_{k+1}+\cdots +x_n \le t_n \rb \\
&& \times \PP \lb B_{> k}^- \cap \lb \Delta
\tau_m(k+1) \in \di x_{k+1}, \dots , \Delta \tau_m(n) \in \di x_n \rb
\mid \mathcal{F}_{\tau_m(k)}^0 \rb \\
&=& \PP \lb A_{> k} \cap B_{> k}^- \mid \mathcal{F}_{\tau_m(k)}^0
\rb .
\end{eqnarray*}

\emph{Step 6.}
Finally, it follows from Step 5 that
\begin{eqnarray*}
\PP \lb A \cap B \rb &=& \PP \lb A_{\le k} \cap A_{>k} \cap
B_{\le k} \cap B_{>k} \rb \\
&=& \EE \left( \EE \left( \II \lb A_{\le k} \cap  B_{\le k} \rb \II
\lb A_{>k} \cap  B_{>k} \rb \mid \mathcal{F}_{\tau_m(k)}^0
\right) \right) \\
&=& \EE \left(  \II \lb A_{\le k} \cap B_{\le k} \rb \PP \lb
A_{>k} \cap B_{>k} \mid \mathcal{F}_{\tau_m(k)}^0 \right) \rb \\
&=& \EE \left(  \II \lb A_{\le k} \cap B_{\le k} \rb \PP \lb
A_{>k} \cap B_{>k}^- \mid \mathcal{F}_{\tau_m(k)}^0 \right) \rb \\
&=& \PP \lb A_{\le k} \cap A_{>k} \cap
B_{\le k} \cap B_{>k}^- \rb .
\end{eqnarray*}
This proves (\ref{eq:indepd}), and so completes the proof of the theorem.
\end{proof}



\end{document}